\documentclass[10pt,psamsfonts]{amsart}
\usepackage{amsmath}
\usepackage{amsthm}
\usepackage{amssymb}
\usepackage{amscd}
\usepackage{amsfonts}
\usepackage{amsbsy}
\usepackage{graphicx}
\usepackage[dvips]{psfrag}
\usepackage{array}
\usepackage{color}
\usepackage{epsfig}
\usepackage{url}
\usepackage{overpic}
\usepackage{epstopdf}
\usepackage[normalem]{ulem}
\usepackage{float}
\usepackage{enumerate}
\usepackage{multirow}
\usepackage{rotating}
\usepackage{hyperref}
\usepackage{float}

\newcommand{\R}{\ensuremath{\mathbb{R}}}
\newcommand{\N}{\ensuremath{\mathbb{N}}}

\newcommand{\CF}{\ensuremath{\mathcal{F}}}
\newcommand{\CG}{\ensuremath{\mathcal{G}}}
\newcommand{\CL}{\ensuremath{\mathcal{L}}}

\newcommand{\CH}{\ensuremath{\mathcal{H}}}
\newcommand{\CO}{\ensuremath{\mathcal{O}}}
\newcommand{\CP}{\ensuremath{\mathcal{P}}}
\newcommand{\U}{\ensuremath{\mathcal{U}}}
\newcommand{\V}{\ensuremath{\mathcal{V}}}

\newcommand{\al}{\alpha}

\newcommand{\s}{\Sigma}

\newcommand{\de}{\delta}
\newcommand{\T}{\theta}
\newcommand{\Span}{\mathrm{Span}}

\def\p{\partial}
\def\e{\varepsilon}

\newtheorem {theorem} {Theorem}%[section]

\newtheorem {proposition} [theorem]{Proposition}

\newtheorem {lemma}  [theorem]{Lemma}

\newtheorem {mtheorem} {Theorem}

\usepackage{mathpazo}

\begin{document}
\renewcommand{\arraystretch}{1.5}
\title[Hilbert's number for PPLVF with an algebraic curve of discontinuity]{On the Hilbert number for piecewise linear vector fields\\ with algebraic discontinuity set}
\author[D.D. Novaes]{Douglas D. Novaes}
\address{Departamento de Matem\'{a}tica - Instituto de Matem\'{a}tica, Estat\'{i}stica e Computa\c{c}\~{a}o Cient\'{i}fica (IMECC) - Universidade
Estadual de Campinas (UNICAMP),  Rua S\'{e}rgio Buarque de Holanda, 651, Cidade Universit\'{a}ria Zeferino Vaz, 13083-859, Campinas, SP,
Brazil} 
\email{ddnovaes@unicamp.br}



\maketitle

\begin{abstract}
The second part of the Hilbert's sixteenth problem consists in determining the upper bound $\CH(n)$ for the number of limit cycles that planar polynomial vector fields of degree $n$ can have.  For $n\geq2$, it is still unknown whether $\CH(n)$ is finite or not. The main achievements obtained so far establish lower bounds for $\CH(n)$. Regarding asymptotic behavior, the best result says that $\CH(n)$ grows as fast as $n^2\log(n)$. Better lower bounds for small values of $n$ are known in the research literature. In the recent paper ``Some open problems in low dimensional dynamical systems'' by A. Gasull, Problem 18 proposes another Hilbert's sixteenth type problem, namely improving the lower bounds for $\CL(n)$, $n\in\N$, which is defined as the maximum number of limit cycles that planar piecewise linear differential systems with two zones separated by a branch of an algebraic curve of degree $n$ can have. So far, $\CL(n)\geq [n/2],$ $n\in\N$, is the best known general lower bound.  Again, better lower bounds for small values of $n$ are known in the research literature. 
Here, by using a recently developed second order Melnikov method for nonsmooth systems with nonlinear discontinuity manifold, it is shown that $\CL(n)$ grows as fast as $n^2.$ This will be achieved by providing lower bounds for $\CL(n)$, which improves every previous estimates for $n\geq 4$.
\end{abstract}

\section{Introduction and statement of the main result}

The second part of the Hilbert's sixteenth problem consists in determining a uniform upper bound for the number of limit cycles that planar polynomial vector fields of a given degree $n$ can have. In this context, the so-called {\it Hilbert Number}  $\CH(n)$ denotes this maximum number of limit cycles. Notice that $\CH(1)=0$, because planar linear vector fields do not admit limit cycles. However, for $n\geq2$, it is still unknown whether $\CH(n)$ is finite or not.  The main achievements obtained so far consist in establishing lower bounds for $\CH(n)$. At an early stage of investigation of this problem, several authors established that $\CH(n)$ grows as fast as $n^2$ (see, for instance, Otrokov \cite{O54}, Il'yashenko \cite{I91}, and Basarab-Horwath \& Lloyde \cite{BL88}). 
In \cite{CL95}, Christopher \& Lloyd showed that $\CH(n)$ is at least $\CO(n^2\log n)$ as $n\to\infty$ (see also \cite{HL12}). For $n\in\N$ small, the current best known lower bounds for $\CH(n)$ was provided by Prohens \& Torregrosa in \cite{PT19} (for previous estimates, see the references therein).

Hilbert's sixteenth problem has also been considered in the context of planar piecewise polynomial vector fields (see, for instance, \cite{CNT19,GouTor20} and the references therein).  Providing upper bounds for the maximum number of limit cycles in this context has also been shown to be very challenging, even in the linear case. Again, the main achievements concern lower bounds. In 2010, Han \& Zhang \cite{HanZhang10} constructed examples of piecewise linear vector fields with two zones separated by a straight line having two limit cycles. They conjectured that such systems cannot have more than 2 limit cycles. In 2012, Huan \& Yang  \cite{HuanYang12} provided an example where 3 limit cycles could be numerically observed. In the same year, Llibre \& Ponce \cite{LP12} proved analytically the existence of such numerically observed limit cycles, providing then a negative answer to the conjecture. After that, several works came out with examples presenting 3 limit cycles  (see, for instance, \cite{BuzziEtAl13,cardoso20,FreireEtAl14,FreireEtAl14b,LlibreEtAl15b,NovaesTorregrosa17}).

In 2014, Braga \& Mello \cite{BragaMello14} showed that the number of limit cycles in planar piecewise linear vector fields strongly depends on the ``shape'' of the discontinuity curve. They conjectured the existence of examples having an arbitrary number of limit cycles. In 2015, Novaes \& Ponce \cite{NovaesPonce15} gave a positive answer to this conjecture. Hence, it has become natural to inquire about the maximum number of limit cycles that planar piecewise linear vector fields with two zones separated by a branch of an algebraic curve of degree $n$ can have. In this context, $\CL(n)$ denotes this maximum number of limit cycles.

Recently, in \cite{GTZ20}, Gasull et al. showed that $\CL(n)\geq [n/2],$ for $n\in\N$. For small values of $n,$ the current best lower bounds for $\CL(n)$ have been provided by Andrade et al. in \cite{ACCN21}, where the Melnikov method at arbitrary order was developed for a class of piecewise smooth systems and then used up to order 6 to show that $\CL(2)\geq 4,$ $\CL(3)\geq 8,$ $\CL(n)\geq7,$ for $n\geq 4$ even, and  $\CL(n)\geq9,$ for $n\geq 5$ odd.

Problem 18 of the recent paper \cite{Gasull2021}, ``Some open problems in low dimensional dynamical systems'' by A. Gasull, proposes improving, if possible, the previous lower bounds for $\CL(n)$. Thus, motivated by that, the main result of this paper improves all the previous lower bounds for $\CL(n)$, with $n\geq 4$:
\begin{mtheorem}\label{main}
Let $k\in\N$, $k>1,$ then  $\CL(2k)\geq k^2+2k+1$ and $\CL(2k+1)\geq k^2+2k+3.$
\end{mtheorem}

Notice that, for small values of $n$, one has $\CL(5)\geq11$, $\CL(6)\geq16$, $\CL(7)\geq18$, $\CL(8)\geq25$, $\CL(9)\geq27$, $\CL(10)\geq36$, $\CL(11)\geq38,$ and so on. In general, $\CL(n)\geq C\, n^2,$ for every $n\in\N$, where $C$ is a positive constant, which means that $\CL(n)$ grows as fast as $n^2$. 

Theorem \ref{main} will be proven in Section \ref{sec:proof} by performing a second order Melnikov analysis on the the following family of planar piecewise linear vector fields:
\begin{equation*}\label{eq:main}
\left\{ \begin{array}{cc}\left(\!\!\!\begin{array}{c}
y+\e\,y+\e^2\al\\
-x-\e\,\beta\,y+\e^2\,\gamma\, y
\end{array}\!\!\!\right), & y> y-x^{2k+1}-\e H_k(x,y), \vspace{0.3cm}\\ 
\left(\!\!\!\begin{array}{c}
y+\e\,\beta\,x\\
\mu-x
\end{array}\!\!\!\right), & y< y-x^{2k+1}-\e H_k(x,y), 
\end{array}
\right.
\end{equation*}
where $H_k$ is a polynomial function of degree $2k$ or $2k+2$. Notice that the discontinuity curve is an algebraic variety of degree $2k+1$ or $2k+2$ given by $$\Sigma_{\e}^k=\{(x,y)\in\R^2:\, y-x^{2k+1}-\e H_k(x,y)=0\}.$$
Here, Filippov's convention \cite{filippov2013differential} is assumed for its trajectories.
 Section \ref{sec:melnikov} is devoted to discuss the Melnikov method for detecting periodic solutions of piecewise smooth differential systems (PSDS).

\section{Periodic solutions in PSDS with nonlinear discontinuity manifold}\label{sec:melnikov}

The {\it Averaging Method} is a classical tool to approach the problem of bifurcation of periodic solutions for differential systems given in the following standard form
\[
\dot x=\e F(t,x,\e),\quad (t,x,\e)\in\R\times \overline{D}\times[-\e_0,\e_0],
\] 
where $D\subset\R^n$ is an open bounded set, $\e$ is assumed to be a small parameter, and $F$ is $T$-periodic in the variable $t$ (see \cite{BuiLli2004,CLN17,LNT2014,SVM}). This method has recently been developed for several classes of PSDS (see, for instance, \cite{ILN2017,LMN15,LliNovRod2017,LNT15}). In these previous studies, when dealing with higher order perturbations, the considered classes of PSDS assume strong conditions on the discontinuity set, for instance, that it is a hyper-surface. Llibre et al. observed in \cite{LMN15} that the first order averaged function,
\[
f_1(x)=\int_0^T F(t,x,0)dt,
\]
always controls the bifurcation of isolated periodic solutions, however higher order averaged functions do not. Recently, the bifurcation functions for detecting periodic solutions of a quite general class of PSDS with nonlinear discontinuity manifold was developed up to order $2$ by Bastos et al. in \cite{BBLN19} and, afterwards, at any order by Andrade et al. in \cite{ACCN21}. Those function are also called Poincar\'{e}-Pontryagin-Melnikov functions or just Melnikov functions. This section is devoted to discuss the second-order Melnikov function.

Let $D\subset\R^d$ be an open bounded subset and $\mathbb{S}^1=\R/T\mathbb{Z},$ for some period $T>0$. Let $\T_j:D\to \mathbb{S}^1$, $j\in\{1,\ldots,N\}$, be $C^{1}$ functions such that $\T_0(x)\equiv0<\T_1(x)<\cdots<\T_N(x)<T\equiv\T_{N+1}(x)$, for all $x\in \overline{D}$. Under the assumptions above, consider the following PSDS defined on $\mathbb{S}^1\times \overline{D}$:
\begin{equation}\label{general-system2} 
\dot{x} = \e F_1(t,x)+\e^2 F_2(t,x)+\e^{3}R(t,x,\e),
\end{equation}
where, for $i\in\{1,2\}$,
\[
F_i(t,x)=\left\{
\begin{array}{ll}
F_i^0(t,x), & 0<t<\T_1(x), \\
F_i^1(t,x), & \T_1(x)<t<\T_2(x), \\
\vdots & \\
F_i^N(t,x), & \T_N(x)<t<T, 
\end{array}
\right.
\]
and
\[
R(t,x,\e)=\left\{
\begin{array}{ll}
R^0(t,x,\e), & 0<t<\T_1(x), \\
R^1(t,x,\e), & \T_1(x)<t<\T_2(x), \\
\vdots & \\
R^N(t,x,\e), & \T_N(x)<t<T, 
\end{array}
\right.
\]
with $F_i^j:\mathbb{S}^1\times \overline{D}\rightarrow\R^d$ and $R^j:\mathbb{S}^1\times \overline{D}\times (-\e_0,\e_0)\rightarrow\R^d$, for $i\in\{1,2\}$ and $j\in\{1,\ldots,N\}$, being $C^{1}$ functions. In this case, the discontinuity manifold is given by $$\s=\{(\T_i(x),x);\ x\in \overline{D},\ i\in\{0,1,\ldots,N\} \}.$$ 

The first and second-order Melnikov functions, $M_1,M_2:D\to\R^d,$ for PSDS \eqref{general-system2} are given as follows:
\begin{equation}\label{mel1}
M_1(x)=f_1(x)\quad \text{and}\quad M_2(x)= f_2(x) + f_2^*(x),
\end{equation}
where 
\begin{equation}\label{prom}
\begin{array}{l}
\displaystyle f_1(x)=\int_{0}^{T}F_1(s,x)ds,\vspace{0.3cm}\\
\displaystyle  f_2(x)= \int_{0}^{T}\bigg[D_xF_1(s,x)
\int_{0}^{s}F_1(t,x)dt+F_2(s,x)\bigg]ds,
\end{array}
\end{equation}
and
\[
f_2^*(x)  = \sum_{j=1}^M \Big(F_1^{j-1}(\T_j(x),x) -
F_1^{j}(\T_j(x),x)\Big)
D_x\T_j(x)\int_0^{\T_j(x)}\hspace{-0.3cm}F_1(s,x)ds.
\]
Notice that $M_1$ coincides with the first-order averaged function and $M_2$ is given by the second-order averaged function $f_2(x)$ plus an increment $f_2^*(x)$, which depends on the ``jump of discontinuity'' and on the ``geometry'' of the discontinuity manifold.

\begin{theorem}[{\cite[Theorem B]{BBLN19}}] \label{melnikov} The following statements hold.
\begin{itemize} 
\item[{\bf i.}] {\bf (First Order)}
Assume that $x^*\in D$ satisfies $M_1(x^*)=0$ and
$\det\left(DM_1(x^*)\right)\neq0.$ Then, for $|\e|\neq0$
sufficiently small, there exists a unique $T$-periodic solution
$x(t,\e)$ of system \eqref{general-system2} such that $x(0,\e)\rightarrow
x^*$ as $\e\rightarrow0.$

\smallskip

\item[{\bf ii.}] {\bf (Second Order)}
Assume that $M_1=0$ and that $x^*\in D$ satifies
$M_2(x^*)=0$ and $\det\left(DM_2(x^*)\right)\neq0.$
Then, for $|\e|\neq0$ sufficiently small, there exists a unique
$T$-periodic solution $x(t,\e)$ of system \eqref{general-system2} such
that $x(0,\e)\rightarrow x^*$ as
$\e\rightarrow0.$
\end{itemize}
\end{theorem}

\section{Proof of the main result}\label{sec:proof}

This section is devoted to the proof of Theorem \ref{main}. It starts by presenting two important concepts which will be used in its proof, namely, the Extended Complete Chebyshev systems (ECT-systems) and the Pseudo-Hopf bifurcation. Section \ref{prelim} provides a useful preliminary result. Theorem \ref{main} is then proven in Section \ref{proofodd} for $n$ odd, and in Section \ref{proofeven} for $n$ even.

\subsection{ECT-systems}

Let $\CF=\left[f_{0}, \ldots, f_{n}\right]$ be an ordered set of smooth functions defined on an interval $I$ and let $\Span(\CF)$ be the set of all linear combinations of elements of $\CF$.
The set $\CF$ is said to be an {\it Extended Complete Chebyshev} system or an ECT-system on $I$ if, and only if, for each $0\leq {\ell}\leq n,$ $W\big(f_0(x),\ldots,f_{\ell}(x)\big)\neq 0$ for every $x\in I$ (see \cite{KASTU1966}). Recall that $W\big(f_0(x),\ldots,f_{\ell}(x)\big)$  denotes the Wronskian of the ordered set $[f_0, \ldots, f_{\ell}]$, that is
\begin{equation*}\label{wrons}
W(f_0(x),\ldots, f_{\ell}(x))(x)=\det\big(M(f_0(x), \ldots, f_{\ell}(x))\big),
\end{equation*}
where
\[
M(f_0, \ldots, f_{\ell})(x)=\left( \begin{array}{ccc}
f_0(x)& \ldots& f_{\ell}(x)\\
f'_0(x)& \ldots & f'_{\ell}(x)\\
\vdots& & \vdots\\
f_0^{({\ell})}(x) &&  f_{\ell}^{({\ell})}(x) 
\end{array}\right).
\]
In particular, if $\CF$ is an ECT-system on $I$, then for each configuration of $m\leq n$ zeros on $I$, taking into account their multiplicity, there exists a function in $\Span(\CF)$ realizing this configuration (see, for instance, \cite[Theorem 1.3]{NovaesTorregrosa17}). This is the main property of ECT-systems that is going to be used in the proof of Theorem \ref{main}.

As an example, one can see that the set $\{x^{a_1},x^{a_2},\ldots, x^{a_n}\}$ is an ECT-system on $I=(0,+\infty)$, provided that $a_i\neq a_j$ for $i,j\in\{1,\ldots,n\}$ such that $i\neq j$. Indeed, for $\ell\in\{1,\ldots,n\}$,
\begin{equation}\label{wromon}
\begin{aligned}
W(x^{a_1},\ldots,x^{a_{\ell}})= \left(\prod_{i=1}^{\ell}\prod_{j=i+1}^{\ell}(a_j-a_i) \right) x^{\rho_{\ell}}\neq0, \, x\in I\, \text{ where }\,\rho_{\ell}=\sum_{i=1}^{\ell} a_i-\dfrac{\ell(\ell-1)}{2}.
\end{aligned}
\end{equation}
The Wronskian above can be computed by using the following  properties of Wronskians (see \cite{hartman69,swia71}):
Let $g:I\to\R$ be a smooth function such that $g(x)\neq0$ for $x\in I$, then:
\begin{itemize}
\item
$W\big(g(x)\,f_0(x),\ldots,g(x)\,f_n(x)\big)=g(x)^{n+1}W\big(f_0(x),\ldots,f_n(x)\big)$ for $x\in I.$

\medskip

\item $W\big(g(x),f_0(x),\ldots,f_n(x)\big)=g(x)^{n+2}W\left(\Big(\dfrac{f_0(x)}{g(x)}\Big)',\ldots,\Big(\dfrac{f_n(x)}{g(x)}\Big)'\right)$ for $x\in I.$
\end{itemize}

\subsection{Pseudo-Hopf bifurcation}
For $\mu\in\R$, consider the following family of piecewise smooth vector fields
\[
Z_{\mu}(x,y)=\begin{cases}
X_{\mu}(x,y)& h(x,y)>0,\\
Y_{\mu}(x,y)& h(x,y)<0,
\end{cases}
\] 
where $h:\R^2\to\R$ is a smooth function having $0$ as a regular value such that $h(0,0)=0$.
Consider the following qualitative conditions on $Z_{\mu}$: 
\begin{itemize}
\item[{\bf H1.}] for $\mu=0$, the origin is either an asymptotically stable or unstable monodromic singularity of $Z_0$;
\item[{\bf H2.}] for $\mu\neq0$, $Z_{\mu}$ has a sliding segment containing the origin which changes stability as $\mu$ changes sign.
\end{itemize}
Notice that, in {\bf H1}, the singularity can be either of focus-focus, focus-tangential, or tangential-tangential type and the stability property can be ensured by means of the Lyapunov coefficients (see, for instance, \cite{gassulcoll,novaessilva21}). Condition ${\bf H2}$ is equivalent to impose that, for each $\mu\neq0$ small, there exists a neighbourhood $U_{\mu}\subset\R^2$ of the origin such that $X_{\mu} h(p)Y_{\mu}h(p)<0$ for every $p\in U_{\mu}\cap\Sigma$ and that the function $\mu\mapsto X_{\mu} h(p)$ changes sign as $\mu$ changes its sign.

Assuming conditions above, the piecewise smooth vector fields $Z_{\mu}$ undergoes a pseudo-Hopf bifurcation at $\mu=0$ (eventually degenerated), giving birth to a sliding segment and a limit cycle, which exists either for $\mu>0$ or $\mu<0$ (see Figure \ref{fig:pseudohopf}).
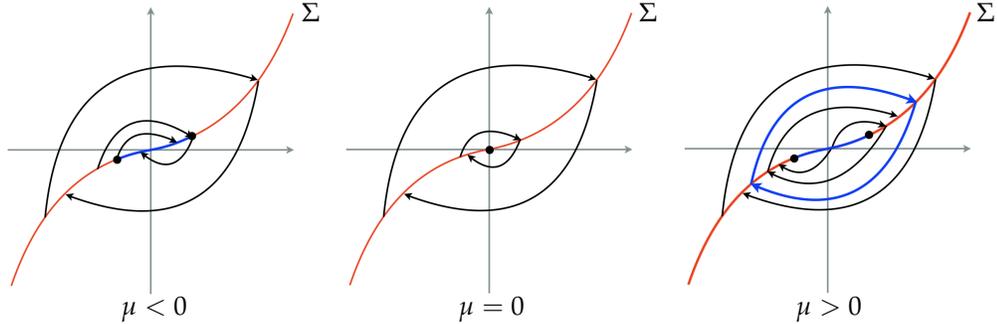
\begin{figure}[H]
    \begin{center}
    %\begin{overpic}[grid,tics=3,width=14cm]{Fig1.png}
       \begin{overpic}[width=13cm]{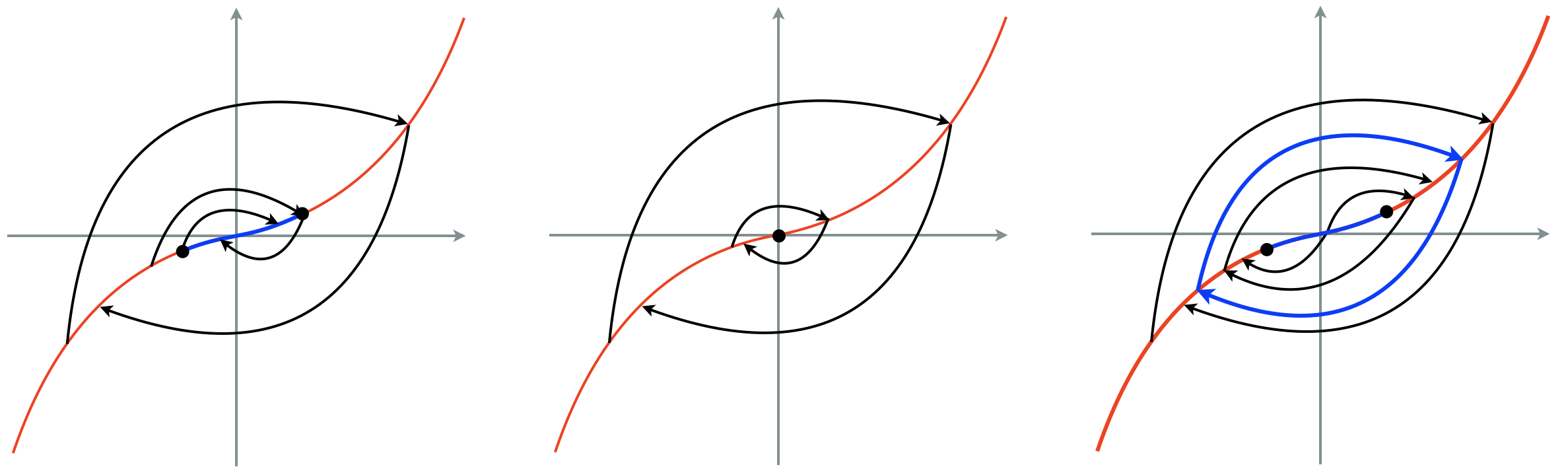}
       \put(30.5,28){$\Sigma$}
         \put(65,28){$\Sigma$}
           \put(99.5,28){$\Sigma$}
           \put(46.5,-2){$\mu=0$}
           \put(12,-2){$\mu<0$}
           \put(81,-2){$\mu>0$}
        \end{overpic}
        
        \vspace{0.5cm}
        
        \caption{Representation of a pseudo-Hopf bifurcation in a piecewise smooth vector field $Z_\mu$ which has a asymptotically stable monodromic singularity at the origin for $\mu=0$; a stable sliding segment containing the origin for $\mu<0$; and a unstable sliding segment for $\mu>0$. In this case, a asymptotically stable limit cycle appears for $\mu>0$.}\label{fig:pseudohopf}
    \end{center}
\end{figure}

The pseudo-Hopf bifurcation was considered by Filippov in his book \cite{filippov2013differential} (see item b of page 241) and has been explored in \cite{castillo17}. When working with cyclicity problems, this bifurcation provides a usefull mechanism to improve the number of limit cycles (see, for instance, \cite{CNT19,GouTor20,novaessilva21}).

\subsection{A preliminary result}\label{prelim}
Let $m$ be a positive integer. Consider an increasing finite sequence of positive integers $p_i,$ for $i\in\{1,\ldots,m\},$ and a sequence of real numbers $c_i,$ for $i\in\{1,\ldots,m\}.$ Let $h_{m}(x;c)$ be the following polynomial function on $x$ of degree $2p_m$:
\[
h_m(x;c)=c_1 x^{2 p_1}+c_2 x^{2p_2}+\cdots+c_m x^{2p_m},
\]
where $c=(c_1,\ldots,c_m)$.
Now, consider the following family of  planar piecewise linear vector fields: 
\begin{equation}\label{general-system1}
Z_{\e}(x,y;\Lambda,\mu)=\left\{ \begin{array}{lc}X_{\e}(x,y;\Lambda,\mu), & y>x^{2k+1}+\e\, h_m(x;c)+\e^2\CO(|c|^2)+\CO(\e^3), \\ 
Y_{\e}(x,y;\Lambda,\mu), & y<x^{2k+1}+\e\, h_m(x;c)+\e^2\CO(|c|^2)+\CO(\e^3), 
\end{array}
\right.
\end{equation}
where
\begin{equation}\label{XY}
X_{\e}(x,y;\Lambda,\mu)=
\left(\!\!\!\begin{array}{c}
y+\e\,y+\e^2\al+\CO(\e^3)\\
-x-\e\,\beta\,y+\e^2\,\gamma\, y+\CO(\e^3)
\end{array}\!\!\!\right),\quad
Y_{\e}(x,y;\Lambda,\mu)=\left(\!\!\!\begin{array}{c}
y+\e\,\beta\,x +\CO(\e^3)\\
\mu-x+\CO(\e^3)
\end{array}\!\!\!\right),
\end{equation}
 $\Lambda=(\al,\beta,\gamma,c_1,\ldots,c_m)\in \R^{3+m},$ and $\mu\in\R$. Assume that 
\begin{equation}\label{liecond}
X_{\e}(0,0;\Lambda,0)=(\e^2\al,0)\,\,\text{ and }\,\,Y_{\e}(0,0;\Lambda,0)=(0,0),
\end{equation}
for every $\Lambda$ and $\e$.

\begin{proposition}\label{prop:funda}
Let $U\subset D,$ $\U\subset\R^{3+m},$ and $I\subset\R$ be neighbourhoods of $(0,0)\in\R^2,$ $(0,\ldots,0)\in\R^{3+m}$, and $0\in\R$, respectively. Then, there exists $(\Lambda_0,\mu_0)\in\U\times I$ such that, for $|\e|\neq0$ sufficiently small, the vector field $Z_{\e}(x,y;\Lambda_0,\mu_0)$ has $m+3$  (crossing) limit cycles inside $U$.
\end{proposition}
\begin{proof}
First, take $\mu=0$. Notice that, given a compact neighbourhood $\U\subset\R^{3+m}$ of $(0,\ldots,0)\in\R^{3+m},$ there exist $\e_0>0$ and a neighbourhood $U_0\subset U\subset \R^2$ of $(0,0)$ such that the trajectories of $Z_{\e}(\cdot;\Lambda,0)\big|_{U_0}$, apart from the origin, cross transversally the discontinuity curve, for every $\Lambda\in\U$ and $\e\in[-\e_0,\e_0]$. Indeed, denote $$f_{\e,\Lambda}(x)=x^{2k+1}+\e\, h_m(x;c)+\e^2\CO(|c|^2)+\CO(\e^3)$$ such that $$(x,y)\in\Sigma_{\e}\,\,\Leftrightarrow\,\,y=f_{\e,\Lambda}(x).$$ Define
\[
\begin{aligned}
&\CL_X(x;\e,\Lambda)=\big\langle\big(-f'_{\e,\Lambda}(x),1\big),X_{\e}(x,f_{\e,\Lambda}(x);\Lambda,0)\big\rangle\,\,\text{ and }\\
&\CL_Y(x;\e,\Lambda)=\big\langle\big(-f'_{\e,\Lambda}(x),1\big),Y_{\e}(x,f_{\e,\Lambda}(x);\Lambda,0)\big\rangle.
\end{aligned}
\]
Notice that
\[
\CL_X(0;0,\Lambda)=\CL_Y(0;0,\Lambda)=0\,\,\text{ and }\,\, \dfrac{\p \CL_X}{\p x}(0;0,\Lambda)=\dfrac{\p \CL_Y}{\p x}(0;0,\Lambda)=-1\,\,\text{ for every }\,\, \Lambda\in\U.
\]
Thus,
by the implicit function theorem and taking into account the compactness of $\U$, there exist $\e_0>0$ and  unique functions $\xi_X,\xi_Y:[-\e_0,\e_0]\times\U \to \R$ such that 
\[
\xi_X(0)=\xi_Y(0)=0\,\text{  and  }\,
\CL_X(\xi_X(\e,\Lambda);\e,\Lambda)=\CL_Y(\xi_Y(\e,\Lambda);\e,\Lambda)=0,
\] for every $(\e,\Lambda)\in [-\e_0,\e_0]\times\U $.
Now, taking \eqref{liecond} into account, one can see that $\CL_X(0;\e,\Lambda)=\CL_Y(0;\e,\Lambda)=0$ for every $(\e,\Lambda)\in[-\e_0,\e_0]\times\U$, which, from the uniqueness of $\xi_X$ and $\xi_Y$, implies that $\xi_X=\xi_Y=0$. This means that there exists an open neighbourhood $J$ of $x=0$ such that $\CL_X(x;\e,\Lambda)<0$ and $\CL_Y(x;\e,\Lambda)<0$ for every $(x,\e,\Lambda)\in (J\setminus\{0\})\times [-\e_0,\e_0]\times\U$. Hence, the claim follows by
denoting $U_0=\{(x,x^{2k+1}:\,x\in J\}$
and taking $\e_0>0$ and $J$ smaller if necessary.

Now, in order to remove the dependence on $\e$ of the discontinuity curve, we perform the change of variables $(u,v)=\big(x,y-\e\, h_m(x;c)+\e^2\CO(|c|^2)+\CO(\e^3)\big)$. In these new variables, the vector field $Z_{\e}(x,y;\Lambda,0)$ writes
\begin{equation}\label{general-system2}
\widetilde Z_{\e}(u,v;\Lambda)=\left\{ \begin{array}{lc}\widetilde X_{\e}(u,v;\Lambda), & v>u^{2k+1}, \\ 
\widetilde Y_{\e}(x,y;\Lambda), & v<u^{2k+1}, 
\end{array}
\right.
\end{equation}
where
\[
\widetilde X_{\e}(u,v)=
\left(\begin{array}{c}
v+\e\big(v+h_m(u;c)\big)+\e^2\big(\al+h_m(u;c)+\CO(|\Lambda|^2)\big)+\CO(\e^3)\\
-u-\e\,v\,\big(\beta+h_m'(u;c)\big)+\e^2\,y\,\big(\gamma-h_m'(u;c)+\CO(|\Lambda|^2)\big)+\CO(\e^3)
\end{array}\right),
\]
and
\[
\widetilde Y_{\e}(u,v)=\left(\begin{array}{c}
v+\e\,\big(\beta\,u+h_m(u;c))+\e^2\CO(|\Lambda|^2)+\CO(\e^3)\\
-u-\e\,v\,h_m'(u;c)+\e^2\CO(|\Lambda|^2)+\CO(\e^3)
\end{array}\right).
\]
In addition, the set $U_0$ is transformed into another neighborhood  $\widetilde U_0$ of the origin, which may depend on $\Lambda$ and $\e$. Nevertheless, one can see that there exists $r_0>0,$ not depending on $\Lambda$ and $\e,$ such that $B_0=\{(u,v):\,|(u,v)|<r_0\}\subset\widetilde U_0$ for every $\Lambda\in\U$ and $\e\in[-\e_0,\e_0]$. Notice that the trajectories of $\widetilde Z_{\e}(\cdot;\Lambda)\big|_{B_0}$, apart from the origin, cross transversally the discontinuity curve $\Sigma=\{(u,v)\in B_0:\,v=u^{2k+1}\}$, for every $\Lambda\in\U$ and $\e\in[-\e_0,\e_0]$.

Now, performing the polar change of variables, $(u,v)=(r\cos\T,r\sin\T),$ and taking $\T$ as the new time variable, the piecewise smooth vector field \eqref{general-system2} is transformed into the following PSDS:
\begin{equation}\label{eq:pol}
\frac{dr}{d\T} = \begin{cases}
F^+(\T,r,\e)=\varepsilon F_1^+(\T,r) + \varepsilon^2 F_2^+(\T,r)+\mathcal{O}(\varepsilon^3),&\sin\T-r^{2k}\cos^{2k+1}\T>0,\\
F^-(\T,r,\e)=\varepsilon F_1^-(\T,r) + \varepsilon^2 F_2^-(\T,r)+\mathcal{O}(\varepsilon^3),&\sin\T-r^{2k}\cos^{2k+1}\T<0,
\end{cases}
\end{equation}
$(\T,r)\in \mathbb{S}^1\times (0,r_0)$, where
\[\begin{aligned}
F_1^+(\T,r)=&r\,h_m'(r \cos \T )\sin ^2\T -  h_m(r \cos \T )\cos \T+r\,(\beta\sin\T-\cos\T)\sin\T,\\
F_1^-(\T,r)=&r\,h_m'(r \cos \T )\sin ^2\T  -  h_m(r \cos \T )\cos \T-r\, \beta \cos ^2\T ,
\end{aligned}
\]
and
\[\begin{aligned}
F_2^-(\T,r)=&2r\cos^2\T\sin^2\T\, h_m'(r\cos\T)-\dfrac{1}{2}(\cos\T+\cos 3\T)h_m(r\cos\T)\\
&+r\cos\T\sin^3\T-\al\,\cos\T+\beta\,r\,\cos2\T\,\sin^2\T-\gamma\,r\,\sin^2\T+\CO(|\Lambda|^2),
\\
F_2^-(\T,r)=&\CO(|\Lambda|^2).
\end{aligned}
\]

In order to understand the transformed discontinuity curve, notice that  $\Sigma$ intersects the circle $u^2+v^2=r^2$ at the
points $(r\cos\T_1(r),r\sin\T_1(r))$ and $(r\cos\T_2(r),r\sin\T_2(r))$ for which the following relationships hold  (see Figure \ref{fig1}):
\begin{equation}\label{eq:id}
r^2=u^2+u^{4k+2}, \,\, \cos\T_1=\frac{u}{r},\,\, \sin\T_1=\frac{u^{2k+1}}{r}, \,\,\text{ and }\,\, \T_1(r)=\T_2(r)+\pi.
\end{equation}
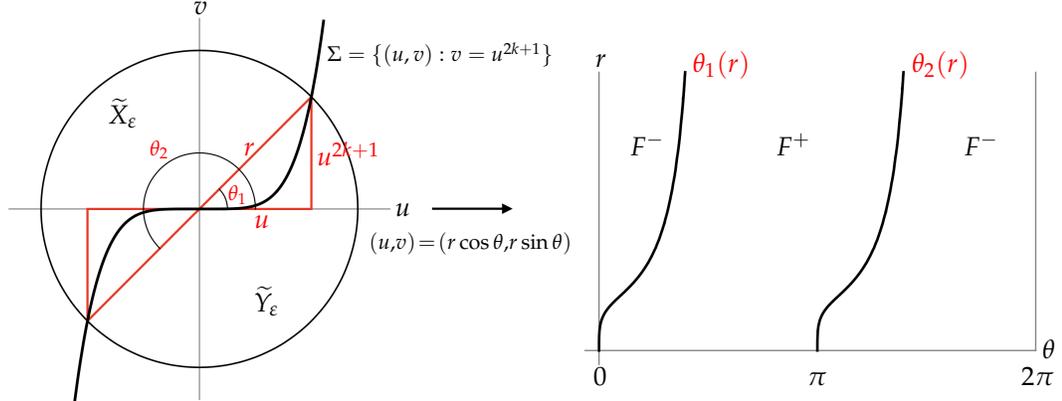
\begin{figure}[H]
    \begin{center}
    %\begin{overpic}[grid,tics=3,height=5.2cm]{Fig0.png}
       \begin{overpic}[height=5.2cm]{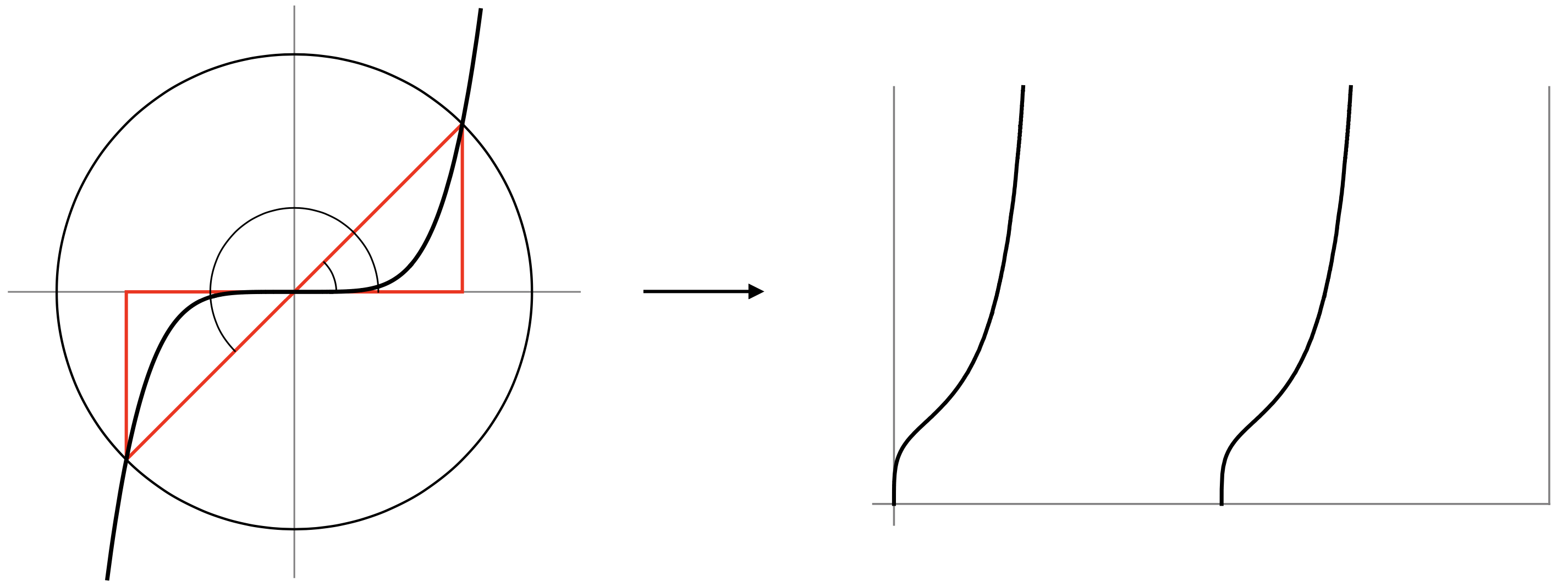}
            \put(21.5,19.7){\footnotesize\color{red}{$\T_1$}}
              \put(14,24){\footnotesize\color{red}{$\T_2$}}
            \put(24,17){\color{red}{$u$}}
            \put(30,23){\color{red}{$u^{2k+1}$}}
            \put(23,24){\color{red}{$r$}}
            \put(31,33){\footnotesize$\Sigma=\{(u,v):v=u^{2k+1}\}$}
            \put(35,15){\footnotesize $(u,\!v)\!=\!(r\cos\T,\!r\sin\T)$}
             \put(66,32){\color{red}{$\T_1(r)$}}
              \put(87,32){\color{red}{$\T_2(r)$}}
                \put(56.7,32.5){$r$}
                  \put(99.5,4.5){$\T$}
                   \put(56.5,2){$0$}
                    \put(97.5,2){$2\pi$}
                     \put(77,2){$\pi$}
                    \put(37.5,18.2){$u$}
                    \put(18.2,37.5){$v$}
                       \put(10,27){$\widetilde X_{\e}$}
                        \put(24,9){$\widetilde Y_{\e}$}
                         \put(60,24){$F^-$}
                         \put(74,24){$F^+$}
                         \put(92,24){$F^-$}
        \end{overpic}
        \caption{In polar coordinates, the discontinuity curve $\Sigma=\{(u,v)\in B_0:\,v=u^{2k+1}\}$ writes   $\{(\T,r)\in \mathbb{S}^1\times(0,r_0]:\, \T=\T_1(r)\,\text{ or }\, \T=\T_2(r)\}$.}\label{fig1}
    \end{center}
\end{figure}
Thus, the PSDS \eqref{eq:pol} can be written in the standard form \eqref{general-system2} as follows
\begin{equation}\label{eq:pol1}
\frac{dr}{d\T} = \begin{cases}
\varepsilon F_1^-(\T,r) + \varepsilon^2 F_2^-(\T,r)+\mathcal{O}(\varepsilon^3),&0<\T<\T_1(r),\\
\varepsilon F_1^+(\T,r) + \varepsilon^2 F_2^+(\T,r)+\mathcal{O}(\varepsilon^3),&\T_1(r)<\T<\T_1(r)+\pi,\\
\varepsilon F_1^-(\T,r) + \varepsilon^2 F_2^-(\T,r)+\mathcal{O}(\varepsilon^3),&\T_1(r)+\pi<\T<2\pi.
\end{cases}
\end{equation}

In what follows, according to Section \ref{sec:melnikov}, the Melnikov functions $M_1,M_2:(0,r_0)\to\R$ will be computed (see the Appendix for the details of the computation).

One can easily see that 
\[
 M_1(r)  = \displaystyle \int_0^{\T_1(r)}  F_1^-(\T,r)\,d\T + \int_{\T_1(r)}^{\T_2(r)}F_1^+(\T,r)\,d\T + \int_{\T_2(r)}^{2\pi}  F_1^-(\T,r)\,d\T=0.
\]

In addition,
\[\label{melnikov2}
M_2(r)= f_2(r) + f_2^*(r),
\]
where
\[\begin{aligned}
f_2(r)=&\displaystyle\int_0^{\T_1(r)}\!\!\left[D_rF_1^-(\T,r) \int_0^\T F_1(\phi,r)\,d\phi+F_2^-(\T,r)\right]\,d\T  \\ 
&+\displaystyle\int_{\T_1(r)}^{\T_2(r)}\!\!\left[D_rF_1^+(\T,r)\int_0^\T F_1(\phi,r)\,d\phi+F_2^+(\T,r)\right]\,d\T  \\ 
&+\displaystyle\int_{\T_2(r)}^{2\pi}\,\left[D_rF_1^-(\T,r)\int_0^\T F_1(\phi,r)\,d\phi+F_2^-(\T,r)\right]\,d\T
\end{aligned}\]
and
\[
\begin{aligned}
f_2^*(r) & =\Big(F_1^-(\T_1(r),r)-F_1^+(\T_1(r),r)\Big)\T_1'(r)\int_0^{\T_1(r)}F_1(\T,r)\,d\T \\
&+\Big(F_1^+(\T_2(r),r)-F_1^-(\T_2(r),r)\Big)\T_2'(r)\int_0^{\T_2(r)}F_1(\T,r)\,d\T.
\end{aligned}
\]
After some manipulation (see the Appendix), if follows that
\[
M_2\big(\sqrt{u^2+u^{4k+2}}\big)=\dfrac{P(u;\Lambda)}{Q(u)},
\]
where
\[
Q(u)=4(1+(1+2k)u^{4k})\sqrt{1+u^{4k}}
\]
and
\[
P(u;\Lambda)=\al\,f_1(u)+\beta\,f_2(u)+\gamma\,f_3(u)+\sum_{i=1}^m c_i g_1(u)+\CO(|\Lambda|^2),
\]
with 
\begin{equation}\label{eq:func}
\begin{aligned}
f_1(u)=&8u^{2k}+8(2k+1)u^{6k},\\
f_2(u)=&-\pi\,u-2\pi(3k+1)u^{4k+1}-\pi(2k+1)u^{1+8k},\\
f_3(u)=&-2\pi\,u-4\pi(k+1)u^{4k+1}-2\pi(2k+1)u^{1+8k},\,\,\text{ and}\\
g_i(u)=&8u^{2(k+p_i)},\,\text{ for }\, i\in\{1,\ldots,m\}.
\end{aligned}
\end{equation}

Now, denote
\[
\CP(u,\Lambda,\de)=\dfrac{1}{\de}P(u;\de\,\Lambda)=\al\,f_1(u)+\beta\,f_2(u)+\gamma\,f_3(u)+\sum_{i=1}^m c_i g_1(u)+\CO(\de^2).
\] 
Consider the sets of $m+3$ functions $\CF=\{f_1,f_2,f_3,g_1,\ldots,g_m\}$ and $\CG=\{h_1,\ldots,h_{m+3}\}$, where
\[
\begin{aligned}
h_1(u)&=\dfrac{2(k+1)f_2(u)-(3k+1)f_3(u)}{4\pi\,k}=u+(2k+1)u^{8k+1},\\
h_2(u)&=\dfrac{f_1(u)}{8}=u^{2k}+(2k+1)u^{6k},\\
h_3(u)&=\dfrac{f_3(u)-2f_2(u)}{8\pi\,k}=u^{4k+1},\\
h_{i+3}(u)&=\dfrac{g_i(u)}{8}=u^{2(k+p_i)},\,\text{ for }\, i\in\{1,\ldots,m\}.
\end{aligned}
\]
Notice that $\Span(\CF)=\Span(\CG)$. In addition, there exists $b_0\in (0,r_0)$ such that $\CG$ is an ECT-system on $(0,b_0)$. Indeed, denote $W_{\ell}(u)=W\big(h_1(u),\ldots,h_{\ell}(u)\big)$ for $\ell\in\{1,\ldots,m+3\}$. Thus,
\[
\begin{aligned}
W_1(u)&=u+(2k+1)u^{8k+1},\\
W_2(u)&=(2k-1)u^{2k}+\CO(u^{6k}),\\
W_3(u)&=4k(4k^2-1)u^{6k-1}+\CO(u^{10k-1}).\\
\end{aligned}
\]
Now, from the linearity of the determinant in each row and taking into account the relationship \eqref{wromon}, one gets
\[
\begin{aligned}
W_{i+3}(u)=&W(u+(2k+1)u^{8k+1},u^{2k}+(2k+1)u^{6k},u^{4k+1},u^{2(k+p_4)},\ldots,u^{2(k+p_{i})})\\
=&W(u,u^{2k},u^{4k+1},u^{2(k+p_1)},\ldots,u^{2(k+p_{i})})\\
&+(2k+1)W(u,u^{6k},u^{4k+1},u^{2(k+p_1)},\ldots,u^{2(k+p_{i})})\\
&+(2k+1)W(u^{8k+1},u^{2k},u^{4k+1},u^{2(k+p_1)},\ldots,u^{2(k+p_{i})})\\
&+(2k+1)^2W(u^{8k+1},u^{6k},u^{4k+1},u^{2(k+p_1)},\ldots,u^{2(k+p_{i})})\\
=&C_{i}\,u^{\rho_{i}}+\CO(u^{\rho_{i}+4k}),
\end{aligned}
\]
where $C_{i}\neq0$ and
\[
\rho_{i}=2\sum_{j=1}^i p_j+(3+i)2k+2-\dfrac{(i+3)(i+2)}{2}>0,
\]
for $i\in\{1,\ldots,m\}.$ Hence,
\[
\lim_{u\searrow 0}\dfrac{W_1(u)}{u}=1,\qquad\lim_{u\searrow 0}\dfrac{W_2(u)}{u^{2k}}=2k-1,\qquad\lim_{u\searrow 0}\dfrac{W_3(u)}{u^{6k-1}}=4k(4k^2-1),
\]
and
\[
\lim_{u\searrow 0}\dfrac{W_{i+3}(u)}{u^{\rho_i}}=C_i,\,\, \text{ for }\,\, i\in\{1,\ldots,m\}.
\]
This implies that there exists $b_0\in(0,r_0)$ such that all the Wronskian $W_1,\ldots, W_{m+1}$ does not vanishes on $(0,b_0)$ and, consequently, $\CG$ is an ECT-system on $(0,b_0)$.

Since $\Span(\CG)=\{\CP(\cdot,\Lambda,0):\,\Lambda\in\R^{3+m}\}$, there exists $\Lambda^*\in\R^{3+m}$ such that $\CP(\cdot,\Lambda^*,0)$ has $m+2$ simple zeros $u_i\in(0,b_0)$, $i\in\{1,\ldots,m+2\}$. 
As a simple application of the implicit function theorem, one can see that, for $\de>0$ sufficiently small,  $\CP(\cdot,\Lambda^*,\de)$ also has $m+2$  simple zeros $u_i(\de)\in(0,b_0)$  for $i\in\{1,\ldots,m+2\}$.   Consequently, by taking $\de=\de^*$ sufficiently small, in order that $\Lambda_0=\de^*\Lambda^*\in \U,$ it follows that $P(\cdot,\Lambda_0)=0$ has $m+2$ simple zeros $u_i^*=u_i(\de^*)\in(0,b_0)$, $i\in\{1,\ldots,m+2\}$.

Theorem \ref{melnikov} implies that the PSDS \eqref{eq:pol1} has $m+2$ isolated $2\pi$-periodic solutions $r_i(\T,\e)$ such that $r_i(\T,0)=r_i^*<b_0<r_0.$ Each solution corresponds to a hyperbolic limit cycle of $\widetilde Z_{\e}(x,y;\Lambda_0)$ inside $B_0.$ Going back through the change of variables, we conclude that  
$Z_{\e}(x,y;\Lambda_0,0)$ has $m+2$ hyperbolic limit cycles inside $U$. 

Finally, the proof follows by noticing that $Z_{\e}(x,y;\Lambda_0,\mu)$ undergoes a pseudo-Hopf bifurcation at $\mu=0$ around the origin, which provides a new limit cycle of small amplitude. Since the previous $m+2$ limit cycles are hyperbolic, there exists $\mu_0\in I$ sufficiently small for which $Z_{\e}(x,y;\Lambda_0,\mu_0)$ has $m+3$ limit cycles. 
\end{proof}

\subsection{Proof of Theorem \ref{main} in the odd degree case}\label{proofodd}
Consider the following family of planar piecewise linear vector fields with two zones separated by a branch of an algebraic curve of degree $2k+1$:
\begin{equation}\label{eq:odd}
Z_{\e}(x,y;\Lambda,\mu)=\left\{ \begin{array}{lc}X_{\e}(x,y;\Lambda,\mu), &  y-x^{2k+1}-\e H_k(x,y;c)>0, \\ 
Y_{\e}(x,y;\Lambda,\mu), &  y-x^{2k+1}-\e H_k(x,y;c)<0, 
\end{array}
\right.
\end{equation}
where 
\[
H_k(x,y;c)=\sum_{j=1}^{k}\sum_{i=0}^{2j}c_{ij} x^i y^{2j-i}
\] 
and $$c=(\underbrace{c_{0,1},c_{1,1},c_{2,1}}_{3},\underbrace{c_{0,2},\cdots,c_{4,2}}_{5},\cdots, \underbrace{c_{0,j}\cdots, c_{2j,j}}_{2j+1},\cdots,\underbrace{c_{0,k},\cdots,c_{2k,k}}_{2k+1})\in\R^{k^2+2k}.$$
Notice that $H_k(x,y;c)$ is a polynomial function on the variables $(x,y)$ of degree $2k$  having only monomials of even degree and $H_k(0,0;c)=0$.

\begin{lemma}\label{lem:odd}
Given compact neighborhoods $U_0\subset D$ and $\V_0\subset\R^{k^2+2k}$ of $(0,0)\in\R^2$ and $(0,\ldots,0)\in\R^{k^2+2k}$, respectively, there exist $\e_0>0$ such that
\[
U_0\cap\Sigma_{\e}^k=\{(x,y)\in U_0:\,y= x^{2k+1}+\e\, h_m(x;c)+\e^2\CO(|c|^2)+\CO(\e^3)\},
\]
for every $(c,\e)\in \V_0\times [-\e_0,\e_0]$, where \[
h_m(x;c)=H_k(x,x^{2k+1};c)=\sum_{j=1}^{k}\sum_{i=0}^{2j}c_{ij} x^{p_{i,j}},\quad p_{i,j}=i+(2k+1)(2j-i)=2(2k+1)j-2k\,i.
\]
In addition, $p_{i_1,j_1}=p_{i_2,j_2}$ if, and only if, $i_1=i_2$ and $j_1=j_2$.
\end{lemma}
\begin{proof}
Let $I_0\subset\R$ be the compact interval for which $(x,x^{2k+1})\in U_0$ for every $x\in I_0$. Define $\CH(x,y;c,\e)=y-x^{2k+1}-\e H_k(x,y;c).$ Notice that 
\[
\CH(x,x^{2k+1};c,0)=0 \,\,\text{ and }\,\,\dfrac{\p\CH}{\p y}(x,x^{2k+1};c,0)= 1,
\]
for every $(x,c)\in I_0\times \V_0$. Thus, taking into account the compactness of $I_0\times\V_0$, the implicitly function theorem provides the existence of $\e_0>0$ and a unique function $\xi:I_0\times\V_0\times[-\e_0,\e_0]\to \R$ such that $\xi(x,c,0)=x^{2k+1}$ and $\CH(x,\xi(x,c,\e);c,\e)=0$ for every $(x,c,\e)\in I_0\times\V_0\times[-\e_0,\e_0].$ In addition, one can easily see that
\[
\begin{aligned}
\xi(x,c,\e)=&x^{2k+1}+\e H_k(x,x^{2k+1};c)+\e^2 H_k(x,x^{2k+1};c)\dfrac{\p H_k}{\p y}(x,x^{2k+1};c)+\CO(\e^3)\\
=&x^{2k+1}+\e H_k(x,x^{2k+1};c)+\e^2 \CO(|c|^2)+\CO(\e^3).
\end{aligned}
\]
Hence, the first part of the result follows by taking $h_m(x;c)=H_k(x,x^{2k+1};c)$.

Now, assume that, for some $j_1,j_2\in\{1,\ldots,k\}$ and $i_1,i_2\in\{0,\ldots,2k\},$ with $i_1\leq 2j_1$ and  $i_2\leq 2j_2$, one has  $p_{i_1,j_1}=p_{i_2,j_2}$, that is, $i_1+(2k+1)(2j_1-i_1)=i_2+(2k+1)(2j_2-i_2)$. If $2j_1-i_1=2j_2-i_2$, then $i_1=i_2$, which also implies that $j_1=j_2$. Thus, assume that $2j_2-i_2>2j_1-i_1$. In this case, $i_2+(2k+1)\big((2j_2-i_2)-(2j_1-i_1)\big)=i_1\leq 2k+1$, which is an absurd, because $i\leq 2k$. Hence, $i_1=i_2$ and $j_1=j_2$.
\end{proof}

From here, the proof of Theorem \ref{main} in the odd degree case follows from Proposition \ref{prop:funda} and by applying Lemma \ref{lem:odd} to the planar piecewise polynomial vector field \eqref{eq:odd}. Indeed, in this cases, Lemma \ref{lem:odd} implies that \eqref{eq:odd}, restricted to $U_0$, writes like \eqref{general-system1} with $m=k^2+2k$. Therefore, from Proposition \ref{prop:funda}, there exists $\Lambda_0\in\U_0=\R^3\times \V_0\subset\R^{k^2+2k+3}$ and $\mu_0\in\R$ such that \eqref{eq:odd} has $k^2+2k+3$ limit cycles in $U_0$. This implies that $\CL(2k+1)\geq k^2+2k+3$.

\subsection{Proof of Theorem \ref{main} in the even degree case}\label{proofeven}
Consider the following family of planar piecewise linear vector fields with two zones separated by a branch of an algebraic curve of degree $2k+2$:
\begin{equation}\label{eq:even}
Z_{\e}(x,y;\Lambda,\mu)=\left\{ \begin{array}{lc}X_{\e}(x,y;\Lambda,\mu), &  y-x^{2k+1}-\e H_k(x,y;c)>0, \\ 
Y_{\e}(x,y;\Lambda,\mu), &  y-x^{2k+1}-\e H_k(x,y;c)<0, 
\end{array}
\right.
\end{equation}
where 
\[
H_k(x,y;c)=\sum_{j=1}^{k+1}\sum_{i=0}^{2j}c_{ij} x^i y^{2j-i}
\] 
and $$c=(\underbrace{c_{0,1},c_{1,1},c_{2,1}}_{3},\underbrace{c_{0,2},\cdots,c_{4,2}}_{5},\cdots, \underbrace{c_{0,j}\cdots, c_{2j,j}}_{2j+1},\cdots,\underbrace{c_{0,k},\cdots,c_{2k+2,k+1}}_{2k+2+1})\in\R^{(k+1)^2+2(k+1)}.$$
Notice that $H_k(x,y;c)$ is a polynomial function on the variables $(x,y)$ of degree $2k+2$  having only monomials of even degree and $H_k(0,0;c)=0$.

\begin{lemma}\label{lem:even}
Given compact neighborhoods $U_0\subset \R^2$ and $\V_0\subset\R^{(k+1)^2+2(k+1)}$ of $(0,0)\in\R^2$ and $(0,\ldots,0)\in\R^{(k+1)^2+2(k+1)}$, respectively, there exist $\e_0>0$
such that
\[
U_0\cap\Sigma_{\e}^k=\{(x,y)\in U_0:\,y= x^{2k+1}+\e\, h_m(x;c)+\e^2\CO(|c|^2)+\CO(\e^3)\},
\]
for every $(c,\e)\in \V_0\times [-\e_0,\e_0]$, where \[
h_m(x;c)=H_k(x,x^{2k+1};c)=\sum_{j=1}^{k+1}\sum_{i=0}^{2j}c_{ij} x^{p_{i,j}},\quad p_{i,j}=i+(2k+1)(2j-i)=2(2k+1)j-2k\,i.
\]
In addition, assuming $i_1\geq i_2$, one has $p_{i_1,j_1}=p_{i_2,j_2}$ if, and only if, one of the conditions hold:
\begin{itemize}
\item[(a)] $i_1=i_2$ and $j_1=j_2$;
\item[(b)] $i_1=2k+1,$ $i_2=0$, $j_1=k+1$,  and $j_2=1$;
\item[(c)] $i_1=2k+2,$ $i_2=1$, $j_1=k+1$,  and $j_2=1$.
\end{itemize}
\end{lemma}
\begin{proof}
The first part of the proof of this lemma is exactly the same as Lemma \ref{lem:odd}.

Assume that, for some $j_1,j_2\in\{1,\ldots,k+1\}$ and $i_1,i_2\in\{0,\ldots,2k+2\},$ with $i_1\leq 2j_1$ and  $i_2\leq 2j_2$, one has  $p_{i_1,j_1}=p_{i_2,j_2}$, that is, $i_1+(2k+1)(2j_1-i_1)=i_2+(2k+1)(2j_2-i_2)$. If $2j_1-i_1=2j_2-i_2$, then $i_1=i_2$, which also implies that $j_1=j_2$, so condition (a) holds. Now, assume  that condition (a) does not hold, in particular, $2j_2-i_2>2j_1-i_1$. In this case, $i_2+(2k+1)\big((2j_2-i_2)-(2j_1-i_1)\big)=i_1\leq 2k+1$. Thus, $(2j_2-i_2)-(2j_1-i_1)=1$ and either $i_2=0$ and $i_1=2k+1$, or $i_2=1$ and $i_1=2k+2.$  For both cases, one has $j_1=j_2+k$ and, since $j_1,j_2\in\{1,\ldots,k-1\}$, implies that $j_2=1$ and $j_1=k+1$.
\end{proof}

From here, the proof of Theorem \ref{main} in the even degree case follows from Proposition \ref{prop:funda} and by applying Lemma \ref{lem:even} to the planar piecewise polynomial vector field \eqref{eq:even}. Indeed, in this cases, Lemma \ref{lem:even} implies that \eqref{eq:even}, restricted to $U_0$, writes like \eqref{general-system1} with $m=(k+1)^2+2(k+1)-2$. This is because, from items (b) and (c) of Lemma \ref{lem:even}, the powers $p_{2k+1,k+1}$ and $p_{2k+2,k+1}$ coincides with $p_{0,1}$ and $p_{1,1},$ respectively, reducing by two the number of distinct monomials with even powers in $h_m$. 

Therefore, from Proposition \ref{prop:funda}, there exists $\Lambda_0\in\U_0=\R^3\times \V_0\subset\R^{(k+1)^2+2(k+1)+1}$ and $\mu_0\in\R$ such that \eqref{eq:even} has $(k+1)^2+2(k+1)+1$ limit cycles in $U_0$. This implies that $\CL(2k+2)\geq (k+1)^2+2(k+1)+1$ or, equivalently, $\CL(2k)\geq k^2+2k+1.$

\section*{Appendix: Computation of the Melnikov functions}
This appendix is devoted to present the details of the computation of the Melnikov functions, $M_1$ and $M_2$, of the PSDS \eqref{eq:pol1}. Recall that,
\begin{equation}\label{F1}
\begin{aligned}
F_1^+(\T,r)=&r\,h_m'(r \cos \T )\sin ^2\T -  h_m(r \cos \T )\cos \T+r\,(\beta\sin\T-\cos\T)\sin\T\\
=&-r\cos\T\,\sin\T+\CO(|\Lambda|),\\
F_1^-(\T,r)=&r\,h_m'(r \cos \T )\sin ^2\T  -  h_m(r \cos \T )\cos \T-r\, \beta \cos ^2\T=\CO(|\Lambda|) ,
\end{aligned}
\end{equation}
and
\begin{equation}\label{F2}
\begin{aligned}
F_2^+(\T,r)=&2r\cos^2\T\sin^2\T\, h_m'(r\cos\T)-\dfrac{1}{2}(\cos\T+\cos 3\T)h_m(r\cos\T)\\
&+r\cos\T\sin^3\T-\al\,\cos\T+\beta\,r\,\cos2\T\,\sin^2\T-\gamma\,r\,\sin^2\T+\CO(|\Lambda|^2),
\\
F_2^-(\T,r)=&\CO(|\Lambda|^2).
\end{aligned}
\end{equation}

\subsection{Computation of $M_1$}

According to Section \ref{sec:melnikov},
\[
 M_1(r)  = \int_0^{\T_1(r)}  F_1^-(\T,r)\,d\T + \int_{\T_1(r)}^{\T_2(r)}F_1^+(\T,r)\,d\T + \int_{\T_2(r)}^{2\pi}  F_1^-(\T,r)\,d\T.
 \]
where $\T_2(r)=\T_1(r)+\pi$ (see the relationships \eqref{eq:id}). By computing the integrals above separately, one has
 \begin{equation}\label{I1}
 \begin{aligned}
\int_0^{\T_1(r)}&  F_1^-(\T,r)\,d\T=\\
=&  \int_0^{\T_1(r)} \Big( r\,h_m'(r \cos \T )\sin ^2\T  -  h_m(r \cos \T )\cos \T-r\, \beta \cos ^2\T \Big)\,d\T\\
=& \int_0^{\T_1(r)} \Big( r\,h_m'(r \cos \T )\sin ^2\T  -  h_m(r \cos \T )\cos \T\Big)\,d\T-r \beta \int_0^{\T_1(r)} \cos ^2\T\, d\T\\
=&\int_0^{\T_1(r)} \Big( r\,h_m'(r \cos \T )\sin ^2\T  -  h_m(r \cos \T )\cos \T\Big)\,d\T-\dfrac{r \beta}{4}(2\T_1(r)+\sin2\T_1(r)),
\end{aligned}
\end{equation}
 \begin{equation}\label{I2}
   \begin{aligned}
\int_{\T_1(r)}^{\T_2(r)}&F_1^+(\T,r)\,d\T=\\
=& \int_{\T_1(r)}^{\T_2(r)} \Big(r\,h_m'(r \cos \T )\sin ^2\T -  h_m(r \cos \T )\cos \T+r\,(\beta\sin\T-\cos\T)\sin\T \Big)\,d\T \\
=& \int_{\T_1(r)}^{\T_2(r)} \Big(r\,h_m'(r \cos \T )\sin ^2\T -  h_m(r \cos \T )\cos \T\Big)\,d\T\\
&+r\, \int_{\T_1(r)}^{\T_2(r)}(\beta\sin\T-\cos\T)\sin\T\,d\T \\
=& \int_{\T_1(r)}^{\T_2(r)} \Big(r\,h_m'(r \cos \T )\sin ^2\T -  h_m(r \cos \T )\cos \T\Big)\,d\T+\dfrac{\pi\,\beta\,r}{2},
\end{aligned}
\end{equation}
 and
 \begin{equation}\label{I3}
    \begin{aligned}
 \int_{\T_2(r)}^{2\pi}&  F_1^-(\T,r)\,d\T=\\
 =&\int_{\T_2(r)}^{2\pi}\Big( r\,h_m'(r \cos \T )\sin ^2\T  -  h_m(r \cos \T )\cos \T-r\, \beta \cos ^2\T \Big)\,d\T\\
 =&\int_{\T_2(r)}^{2\pi}\Big( r\,h_m'(r \cos \T )\sin ^2\T  -  h_m(r \cos \T )\cos \T\Big)-r \beta\int_{\T_2(r)}^{2\pi} \cos ^2\T\,d\T\\
  =&\int_{\T_2(r)}^{2\pi}\Big( r\,h_m'(r \cos \T )\sin ^2\T  -  h_m(r \cos \T )\cos \T\Big)\\
  &-\dfrac{r \beta}{4}(2(\pi-\T_1(r))-\sin2\T_1(r))\,d\T.
 \end{aligned}
 \end{equation}
 Adding up the integrals above, the first order Melnikov function writes
 \[
 M_1(r)  = \int_0^{2\pi}\Big(r\,h_m'(r \cos \T )\sin ^2\T  -  h_m(r \cos \T )\cos \T\Big)\,d\T.
 \]
 Since the integrand above has only terms of kind $\cos^{2p-1}\T\sin^2\T$ and $\cos^{2p+1}\T$ for $p\geq 1$ integer, it follows that $M_1(r)=0$.

\subsection{Computation of $M_2$}

According to Section \ref{sec:melnikov},
\[\label{melnikov2}
M_2(r)= f_2(r) + f_2^*(r),
\]
where
\[\begin{aligned}
f_2(r)=&\int_0^{\T_1(r)}\!\!\left[D_rF_1^-(\T,r) \int_0^\T F_1(\phi,r)\,d\phi+F_2^-(\T,r)\right]\,d\T  \\ 
&+\displaystyle\int_{\T_1(r)}^{\T_2(r)}\!\!\left[D_rF_1^+(\T,r)\int_0^\T F_1(\phi,r)\,d\phi+F_2^+(\T,r)\right]\,d\T  \\ 
&+\displaystyle\int_{\T_2(r)}^{2\pi}\,\left[D_rF_1^-(\T,r)\int_0^\T F_1(\phi,r)\,d\phi+F_2^-(\T,r)\right]\,d\T\\
\end{aligned}\]
and
\[
\begin{aligned}
f_2^*(r) & =\Big(F_1^-(\T_1(r),r)-F_1^+(\T_1(r),r)\Big)\T_1'(r)\int_0^{\T_1(r)}F_1(\T,r)\,d\T \\
&+\Big(F_1^+(\T_2(r),r)-F_1^-(\T_2(r),r)\Big)\T_2'(r)\int_0^{\T_2(r)}F_1(\T,r)\,d\T,
\end{aligned}
\]
with $\T_2(r)=\T_1(r)+\pi$ (see the relationships \eqref{eq:id}). 

First of all, notice that, for $\T\in[0,\T_1(r)],$
 \begin{equation}\label{dFIF1}
 \begin{aligned}
\int_0^{\T}  F_1(\phi,r)\,d\phi=&\int_0^{\T}  F_1^-(\phi,r)\,d\phi\\
=& \int_0^{\T} \Big( r\,h_m'(r \cos \phi )\sin ^2\phi  -  h_m(r \cos \phi )\cos \phi\Big)\, d\phi-r \beta \int_0^{\T} \cos ^2\phi\, d\phi\\
=&\int_0^{\T} \Big( r\,h_m'(r \cos \phi )\sin ^2\phi  -  h_m(r \cos \phi )\cos \phi\Big)\, d\phi-\dfrac{r \beta}{4}(2\T+\sin2\T)\\=&\CO(|\Lambda|);
\end{aligned}
  \end{equation}
for $\T\in[\T_1(r),\T_2(r)],$
 \begin{equation}\label{dFIF2}
 \begin{aligned}
\int_0^{\T}  F_1(\phi,r)\,d\phi=&\int_0^{\T_1(r)}  F_1(\phi,r)\,d\phi +\int_{\T_1(r)}^{\T}F_1^+(\phi,r)\,d\phi\\
=& \int_{0}^{\T} \Big(r\,h_m'(r \cos \phi )\sin ^2\phi -  h_m(r \cos \phi )\cos \phi\Big)\, d\phi\\
&-\dfrac{r \beta}{4}(2\T_1(r)+\sin2\T_1(r))+r\, \int_{\T_1(r)}^{\T}(\beta\sin\phi-\cos\phi)\sin\phi\, d\phi \\
=& \int_{0}^{\T} \Big(r\,h_m'(r \cos \phi )\sin ^2\phi -  h_m(r \cos \phi )\cos \phi\Big)\, d\phi\\
&+\dfrac{r}{4}\left(2\beta\T-\beta\sin 2\T -4\beta\T_1(r)\right)+\dfrac{r}{4}\left(\cos 2\T-\cos2\T_1(r)\right)\\
=&\dfrac{r}{4}\left(\cos 2\T-\cos2\T_1(r)\right)+\CO(|\Lambda|);
\end{aligned}
  \end{equation}
and,  for $\T\in[\T_2(r),2\pi],$
 \begin{equation}\label{dFIF3}
   \begin{aligned}
\int_{0}^{\T}F_1(\phi,r)\,d\phi=&\int_{0}^{\T_2(r)}F(\phi,r)\,d\phi+\int_{\T_2(r)}^{\T}F_1^-(\phi,r)\,d\phi\\
=&\int_{0}^{\T}\Big( r\,h_m'(r \cos \phi )\sin ^2\phi  -  h_m(r \cos \phi )\cos \phi\Big)\,d\phi\\
&-\dfrac{\beta\,r}{4}\left(2\T_1(r)-2\pi+\sin 2\T_1(r)\right)-r \beta\int_{\T_2(r)}^{\T} \cos ^2\phi\, d\phi\\
=&\int_{0}^{\T}\Big( r\,h_m'(r \cos \phi )\sin ^2\phi  -  h_m(r \cos \phi )\cos \phi\Big)\,d\phi\\
&-\dfrac{\beta\,r}{4}\left(2\T-4\pi+\sin2\T\right)=\CO(|\Lambda|).
\end{aligned}
  \end{equation}

Since, from \eqref{F1}, \eqref{F2}, \eqref{dFIF1}, and \eqref{dFIF3}, 
\[
\int_0^{\T_1(r)}\!\!\left[D_rF_1^-(\T,r) \int_0^\T F_1(\phi,r)\,d\phi+F_2^-(\T,r)\right]\,d\T=\CO(|\Lambda|^2)
\]
and 
\[
\int_{\T_2(r)}^{2\pi}\,\left[D_rF_1^-(\T,r)\int_0^\T F_1(\phi,r)\,d\phi+F_2^-(\T,r)\right]\,d\T=\CO(|\Lambda|^2),
\]
thus
\[
f_2(r)=\int_{\T_1(r)}^{\T_2(r)}\!\!\left[D_rF_1^+(\T,r)\int_0^\T F_1(\phi,r)\,d\phi+F_2^+(\T,r)\right]\,d\T.
\]
Now, from \eqref{F1}, \eqref{F2}, and \eqref{dFIF2}, for $\T\in[\T_1(r),\T_2(r)],$ one has
\[
\begin{aligned}
D_rF_1^+(\T,r)\int_0^\T& F_1(\phi,r)\,d\phi+F_2^+(\T,r)=\\
=&\dfrac{r}{8}\sin2\T\,\Big(2-3\cos2\T+\cos2\T_1(r)\Big)-\al\cos\T-r\,\gamma \sin^2\T\\
&+\beta\dfrac{r}{2}\sin\T\Big((2\T_1(r)-\T)\cos\T+\sin\T\big(3\cos2\T+\sin^2\T_1(r)\big)\Big)\\
&-\dfrac{1}{2}(\cos\T+\cos 3\T)h_m(r\cos\T)\\
&+\dfrac{r}{8}\Big(1-3\cos4\T+2\cos 2\T\,\cos2\T_1(r)\Big)h_m'(r \cos \T )\\
&+\dfrac{r^2}{4}\cos\T\,\sin^2\T\,\left(\cos 2\T-\cos2\T_1(r)\right) \,h_m''(r \cos \T )\\
&-\cos\T\sin\T\int_{0}^{\T} \Big(r\,h_m'(r \cos \phi )\sin ^2\phi -  h_m(r \cos \phi )\cos \phi\Big)\, d\phi+\CO(|\Lambda|^2).\\
\end{aligned}
\]
Therefore,
\begin{equation}\label{eq:f2}
\begin{aligned}
f_2(r)=&2\al\sin\T_1(r)-\dfrac{\pi\,r}{4}(\beta+2\gamma)+\int_{\T_1(r)}^{\T_2(r)}\left[-\dfrac{1}{2}(\cos\T+\cos 3\T)h_m(r\cos\T)\right.\\
&+\dfrac{r}{8}\Big(1-3\cos4\T+2\cos 2\T\,\cos2\T_1(r)\Big)h_m'(r \cos \T )\\
&\left.+\dfrac{r^2}{4}\cos\T\,\sin^2\T\,\left(\cos 2\T-\cos2\T_1(r)\right) \,h_m''(r \cos \T )\right]\,d\T\\
&-\int_{\T_1(r)}^{\T_2(r)}\left[\cos\T\sin\T\int_{0}^{\T} \Big(r\,h_m'(r \cos \phi )\sin ^2\phi - h_m(r \cos \phi )\cos \phi\Big)\, d\phi\right]\,d\T\\
&+\CO(|\Lambda|^2).
\end{aligned}
\end{equation}

In order to compute the increment $f_2^*(r)$, notice that
\[
F_1^-(\T_1(r),r)-F_1^+(\T_1(r),r)=r\cos\T_1(r)\sin\T_1(r)-\beta r
\]
and
\[
F_1^+(\T_2(r),r)-F_1^-(\T_2(r),r)=-r\cos\T_1(r)\sin\T_1(r)+\beta r.
\]
Thus,
\[
\begin{aligned}
\Big(F_1^-(\T_1(r),r)&-F_1^+(\T_1(r),r)\Big)\T_1'(r)\int_0^{\T_1(r)}F_1(\T,r)\,d\T\\=&-\beta \dfrac{r^2}{8} \T'_1(r)\sin 2\T_1(r)\left(\sin 2\T_1(r)+2\T_1(r)\right) \\
&+\dfrac{r}{2}\T'_1(r)\sin 2\T_1(r)\int_0^{\T_1(r)} \Big( r\,h_m'(r \cos \T )\sin ^2\T  -  h_m(r \cos \T )\cos \T\Big)\,d\T+\CO(|\Lambda|^2)
\end{aligned}
\]
and
\[
\begin{aligned}
\Big(F_1^-(\T_2(r),r)&-F_2^+(\T_1(r),r)\Big)\T_2'(r)\int_0^{\T_2(r)}F_1(\T,r)\,d\T\\=&\beta \dfrac{r^2}{8} \T'_1(r)\sin 2\T_1(r)\left(-2\pi+\sin 2\T_1(r)+2\T_1(r)\right) \\
&-\dfrac{r}{2}\,\T'_1(r)\sin 2\T_1(r)\int_0^{\T_2(r)} \Big( r\,h_m'(r \cos \T )\sin ^2\T  -  h_m(r \cos \T )\cos \T\Big)\,d\T+\CO(|\Lambda|^2),
\end{aligned}
\]
where
\[
\T'_1(r)=\T'_2(r)=\dfrac{2k\,r^{2k-1}\cos^{2k}\T_1(r)}{1+(2k+1)r^{2k}\cos^{2k-1}\T_1(r)\,\sin\T_1(r)}
\]
or, equivalently, considering the relationships \eqref{eq:id},
\begin{equation}\label{dtheta}
\T'_1\big(\sqrt{u^2+u^{4k+2}}\big)=\dfrac{2 k\, u^{2k+1}}{(1+(1+2k)u^{4k})\sqrt{1+u^{4k}}}.
\end{equation}
Therefore,
\begin{equation}\label{eq:f2star}
\begin{aligned}
f_2^*(r)  =&-\dfrac{r}{2}\T_1'(r)\sin 2\T_1(r)\int_{\T_1(r)}^{\T_2(r)} \Big( r\,h_m'(r \cos \T )\sin ^2\T  -  h_m(r \cos \T )\cos \T\Big)\,d\T\\
&-\beta\dfrac{\pi r^2}{4}\T_1'(r)\sin 2\T_1(r)+\CO(|\Lambda|^2).
\end{aligned}
\end{equation}

From \eqref{eq:f2} and \eqref{eq:f2star}, one has
\[
M_2(r)=M_{21}(r)+M_{22}(r)+\CO(|\Lambda|^2),
\]
where
\[
M_{21}(r)=2\al\sin\T_1(r)-\beta\dfrac{\pi r}{4}\big(1+r\T_1'(r)\sin 2\T_1(r)\big)-\gamma\dfrac{\pi\,r}{2}.
\]
and
\[
\begin{aligned}
M_{22}(r)=&\int_{\T_1(r)}^{\T_2(r)}\left[\dfrac{1}{2}\cos\T\,\big(r\,\T_1'(r)\sin2\T_1(r)-2\cos2\T\big)h_m(r\cos\T)\right.\\
&+\dfrac{r}{8}\Big(1-3\cos4\T+2\cos 2\T\,\cos2\T_1(r)-4r\,\T_1'(r)\sin^2\T\,\sin 2\T_1(r)\Big)h_m'(r \cos \T )\\
&\left.+\dfrac{r^2}{4}\cos\T\,\sin^2\T\,\left(\cos 2\T-\cos2\T_1(r)\right) \,h_m''(r \cos \T )\right]\,d\T\\
&-\int_{\T_1(r)}^{\T_2(r)}\left[\cos\T\sin\T\int_{0}^{\T} \Big(r\,h_m'(r \cos \phi )\sin ^2\phi - h_m(r \cos \phi )\cos \phi\Big)\, d\phi\right]\,d\T.\\
\end{aligned}
\]

By substituting the relationships \eqref{eq:id} and \eqref{dtheta} into $M_{21}(r)$, one obtain that
\begin{equation}\label{eq:M21}
M_{21}\big(\sqrt{u^2+u^{4k+2}}\big)=\dfrac{\al\,f_1(u)+\beta\,f_2(u)+\gamma\,f_3(u)}{4(1+(1+2k)u^{4k})\sqrt{1+u^{4k}}},
\end{equation}
where $f_1,f_2$ and $f_3$ are given by \eqref{eq:func}.

Now, in order to compute $M_{22}(r)$, recall that $h_m(x;c)=c_1 x^{2 p_1}+c_2 x^{2p_2}+\cdots+c_m x^{2p_m}$. Thus,
\[
M_{22}(r)=c_1\mu_1(r)+c_2\mu_2(r)+\cdots+c_m\mu_m(r),
\]
where
\[
\begin{aligned}
\mu_i(r)=&r^{2p_i}\int_{\T_1(r)}^{\T_2(r)}\left[\dfrac{1}{2}\big(r\,\T_1'(r)\sin2\T_1(r)-2\cos2\T\big)\cos^{2p_i+1}\T \right.\\
&+p_i\dfrac{1}{4}\Big(1-3\cos4\T+2\cos 2\T\,\cos2\T_1(r)-4r\,\T_1'(r)\sin^2\T\,\sin 2\T_1(r)\Big)\cos^{2p_i-1}\T\\
&\left.+p_i(2p_i-1)\dfrac{1}{2}\cos\T\,\sin^2\T\,\left(\cos 2\T-\cos2\T_1(r)\right) \cos^{2p_i-2}\T\right]\,d\T\\
&+r^{2p_i}\int_{\T_1(r)}^{\T_2(r)}\left[\cos\T\sin\T\int_{0}^{\T} \cos^{2pi-1}\phi\,\big(\cos^2\phi-2p_i\sin^2\phi\big)\, d\phi\right]\,d\T.\\
\end{aligned}
\]
Notice that
\[
\int_{0}^{\T} \cos^{2pi-1}\phi\,\big(\cos^2\phi-2p_i\sin^2\phi\big)\, d\phi=\cos^{2p_i}\sin\phi,
\]
for every integer $p_i\geq1$. Thus, $\mu_i$ writes
\[
\begin{aligned}
\mu_i(r)=&\dfrac{r^{2p_i}}{4}\big(p_i\cos2\T_1(r)+r\T_1'(r)\,\sin 2\T_1(r)\big)\int_{\T_1(r)}^{\T_2(r)}\cos^{2p_i-1}\T\big(1-2p_i+(1+2p_i)\cos 2\T\big)\,d\T\\
&+\dfrac{r^{2p_i}}{8}\int_{\T_1(r)}^{\T_2(r)}\cos^{2p_i-1}\T\big((3-2p_i)-1+(4p_i^2-2p_i-4)\cos 2\T-(1+p_i)(3+2p_i)\cos 4\T\big)\,d\T.
\end{aligned}
\]
The integrals above can be computed directly for any integer $p_i\geq 1$, providing
\[
\mu_i(r)=r^{2p_i}2\cos^{2p_i+1}\T_1(r)\,\sin \T_1(r)\big(\cos\T_1(r)-r\,\T_1'(r)\,\sin\T_1(r)\big).
\]

Finally, substituting \eqref{eq:id} and \eqref{dtheta} into $\mu_i$, it follows that
\[
\mu_i\big(\sqrt{u^2+u^{4k+2}}\big)=\dfrac{8u^{2(k+p_i)}}{4(1+(1+2k)u^{4k})\sqrt{1+u^{4k}}},
\]
which implies that
\[
M_{22}(\sqrt{u^2+u^{4k+2}}\big)=\dfrac{2}{4(1+(1+2k)u^{4k})\sqrt{1+u^{4k}}}\big(c_1 g_1(u)+\cdots+c_m g_m(u) \big),
\]
where $g_i(u)$, for $i\in\{1,\ldots,m\}$ are given in \eqref{eq:func}. Therefore,
\[
M_2\big(\sqrt{u^2+u^{4k+2}}\big)=\dfrac{P(u;\Lambda)}{Q(u)},
\]
where
\[
Q(u)=4(1+(1+2k)u^{4k})\sqrt{1+u^{4k}}
\]
and
\[
P(u;\Lambda)=\al\,f_1(u)+\beta\,f_2(u)+\gamma\,f_3(u)+\sum_{i=1}^m c_i g_1(u)+\CO(|\Lambda|^2).
\]

\section*{Acknowledgements}

DDN is partially supported by S\~{a}o Paulo Research Foundation (FAPESP) grants 2018/16430-8, 2018/ 13481-0, and 2019/10269-3, and by Conselho Nacional de Desenvolvimento Cient\'{i}fico e Tecnol\'{o}gico (CNPq) grants 306649/2018-7 and  438975/2018-9.

\bibliographystyle{abbrv}
\bibliography{references.bib}

\end{document}